\documentclass[10pt,a4paper]{amsart}

\usepackage{amsfonts,amsthm,amssymb,latexsym,amsmath,setspace,amscd,verbatim}
\usepackage[pdftex]{graphicx}
\usepackage{graphics}
\usepackage[matrix,arrow]{xy}
\usepackage[active]{srcltx}
\usepackage[colorlinks=true]{hyperref}
\usepackage{mathrsfs}

\title[Action is linking]{On the relation between action and linking}
\author[Bechara Senior]{David Bechara Senior}
\author[Hryniewicz]{Umberto L. Hryniewicz}
\author[Salom\~ao]{Pedro A. S. Salom\~ao}
\date{\today}

\address{David Bechara Senior \\ Ruhr-Universit\"at Bochum, Universit\"atsstrasse 150, IB 3/79, Bochum D-44801, Germany}
\email{david.becharasenior@ruhr-uni-bochum.de}
\address{Umberto L. Hryniewicz \\ RWTH Aachen, Jakobstrasse 2, Aachen 52064, Germany}
\email{hryniewicz@mathga.rwth-aachen.de}
\address{Pedro A. S. Salom\~ao, Instituto de Matem\'atica e Estat\'istica, Departamento de Matem\'atica, Universidade de S\~ao Paulo, Rua do Mat\~ao, 1010, Cidade Universit\'aria, S\~ao Paulo SP, Brazil 05508-090}
\email{psalomao@ime.usp.br}
\address{NYU-ECNU Institute of Mathematical Sciences at NYU Shanghai, 3663 Zhongshan Road North, Shanghai, 200062, China}
\email{pas383@nyu.edu}

\newcommand{\C}{\mathbb{C}}
\newcommand{\R}{\mathbb{R}}
\newcommand{\Z}{\mathbb{Z}}

\newcommand{\D}{\mathbb{D}}

\newcommand{\vol}{\mathrm{vol}}
\newcommand{\link}{\mathrm{link}}
\newcommand{\inter}{\mathrm{int}}

\newcommand{\leb}{\mathrm{Leb}}
\newcommand{\cal}{\mathrm{CAL}}
\newcommand{\T}{\mathbb{T}}
\newcommand{\U}{\mathcal{U}}
\newcommand{\supp}{{\rm supp}}

\theoremstyle{plain}
\newtheorem{theorem}{\sc Theorem}[section]

\newtheorem{proposition}[theorem]{\sc Proposition}
\newtheorem{lemma}[theorem]{\sc Lemma}
\newtheorem{corollary}[theorem]{\sc Corollary}

\theoremstyle{definition}
\newtheorem{definition}[theorem]{\sc Definition}

\theoremstyle{remark}
\newtheorem{remark}[theorem]{\sc Remark}

\begin{document}

\begin{abstract}
We introduce numerical invariants of contact forms in dimension three and use asymptotic cycles to estimate them. As a consequence, we prove a version for Anosov Reeb flows of results due to Hutchings and Weiler on mean actions of periodic points. The main tool is the Action-Linking Lemma, expressing the contact area of a surface bounded by periodic orbits as the Liouville average of the asymptotic intersection number of most trajectories with the surface.
\end{abstract}

\maketitle


\section{Introduction}

This paper aims to explore elementary ergodic methods to study contact forms from a quantitative point of view. We introduce numerical invariants of a systolic flavor that provide the language to generalise results due to Hutchings~\cite{mean} and Weiler~\cite{weiler} on mean actions of periodic points. They use embedded contact homology (ECH) to prove their results. Here we explore the equidistribution of periodic orbits and use classical tools in Dynamical Systems, such as Sigmund's results for Axiom A flows~\cite{sigmund}, to prove analogous statements for Anosov Reeb flows in dimension three. We also explain the use of Irie's equidistribution theorem~\cite{Irie_eq} to check the hypotheses of our statements for $C^\infty$-generic Reeb flows.

The paper is organised as follows. The result for Anosov Reeb flows is Theorem~\ref{thm_anosov}. The application to $C^\infty$-generic Reeb flows relying on~\cite{Irie_eq} is explained in Remark~\ref{rmk_Irie_eq}. These applications use Theorem~\ref{thm_eq_2}, which generalises Theorem~\ref{thm_eq_1} from $S^3$ to the general $3$-manifold, and is proved in section~\ref{sec_proof_thm}. The arguments are based on the Action-Linking Lemma (Lemma~\ref{lem_action_linking}) proved via asymptotic cycles in section~\ref{sec_proof_action_link}. The connection to the results from~\cite{mean,weiler} is discussed in subsection~\ref{ssec_HW}. Finally, in section~\ref{sec_toric} we study Reinhardt domains (toric domains).

\subsection{Set-up and main results}

A contact form on a $3$-manifold $M$ is a $1$-form~$\lambda$ such that $\lambda \wedge d\lambda$ defines a volume form. The associated Reeb vector field $X$ is implicitly determined by $i_Xd\lambda = 0$, $i_X\lambda = 1$. We call the pair $(M,\lambda)$ a contact-type energy level. It is always assumed that $M$ is compact, that $X$ is tangent to $\partial M$ in case $\partial M \neq \emptyset$, and that $M$ is oriented by $\lambda \wedge d\lambda$. The flow of $X$ is called the Reeb flow.

A periodic orbit $\gamma$ of the Reeb vector field consists of the data of a Reeb trajectory that is periodic and one of its positive periods, not necessarily the primitive one. The period is denoted by $T(\gamma)$. Sometimes we might think of $\gamma$ as a (possibly) multiply covered, oriented knot that is positively tangent to the Reeb vector field. The contact volume of $\lambda$ is defined as
\begin{equation}
\vol(\lambda) = \int_M \lambda \wedge d\lambda.
\end{equation}
Systolic geometry of contact forms is mainly concerned with the relations between periods of closed Reeb orbits and the contact volume.

\begin{definition}\label{def_syst_invs}
Let $\lambda$ be a contact form on $S^3$. Given two geometrically distinct periodic Reeb orbits $\gamma_1,\gamma_2$ their \textit{systolic pairing} $\rho(\gamma_1,\gamma_2)$ is defined as
\begin{equation}\label{systolic_pairing}
\rho(\gamma_1,\gamma_2) = \frac{\link(\gamma_1,\gamma_2)\vol(\lambda)}{T(\gamma_1)T(\gamma_2)}
\end{equation}
where $\link$ denotes the linking number. The \textit{systolic interval} of $\lambda$ is defined as
\begin{equation}
I(\lambda) = [\rho_-(\lambda),\rho_+(\lambda)] \qquad \rho_-(\lambda) = \inf_{\gamma_1,\gamma_2} \rho(\gamma_1,\gamma_2)\ \ \ \rho_+(\lambda) = \sup_{\gamma_1,\gamma_2} \rho(\gamma_1,\gamma_2)
\end{equation}
where $\sup$ and $\inf$ are taken over pairs of geometrically distinct periodic orbits. The \textit{systolic norm} of $\lambda$ is defined as the length of the systolic interval
\begin{equation}
\|\lambda\|_{\rm sys} = \rho_+(\lambda) - \rho_-(\lambda).
\end{equation}
\end{definition}


Note that $\rho(\gamma_1,\gamma_2)$ is invariant under iterations of periodic orbits and rescaling of the contact form. The existence of at least two periodic orbits is taken for granted in Definition~\ref{def_syst_invs}. The existence of two or more periodic Reeb orbits in a general closed $3$-manifold is proved via ECH in~\cite{CGH}. An argument based on linearised contact homology for the same result in the specific case of the tight $3$-sphere can be found in~\cite{GHHM}.

\begin{remark}[Finiteness of the systolic norm]
Consider any smooth non-singular vector field on $S^3$ and let $\mathscr{P}$ be the set of its invariant Borel probability measures equipped with the weak* topology. In~\cite{ghys} Ghys introduced and briefly sketched the theory of the quadratic linking form $Q:\mathscr{P} \times\mathscr{P} \to \R$. See also the lecture notes~\cite{dehornoy} by Dehornoy. It turns out that if $\gamma_j:\R/T(\gamma_j)\Z \to S^3$ ($j=1,2$) are geometrically distinct periodic orbits
then
\[
Q\left( \frac{(\gamma_1)_*\leb}{T(\gamma_1)} , \frac{(\gamma_2)_*\leb}{T(\gamma_2)} \right) = \frac{\link(\gamma_1,\gamma_2)}{T(\gamma_1)T(\gamma_2)}
\]
where $\leb$ denotes Lebesgue measure. This identity follows from the definition of the quadratic linking form explained in~\cite[subsection~3.1, page 52]{ghys} in the case of two ergodic measures not supported on the same periodic orbit. In the same reference it is stated that $Q$ is continuous. From the continuity of $Q$ and the compactness of $\mathscr{P}$ it follows that $Q$ is bounded from above and below. This proves the following statement: {\it ``The continuity of the quadratic linking form implies that $\rho_+(\lambda)$ and $\rho_-(\lambda)$ are finite.''} We do not address existence or continuity of the quadratic linking form here, we only note that in the light of~\cite{ghys} one expects all contact forms on $S^3$ to have finite systolic norm. Our results do not use this fact.
\end{remark}

It is instructive to study the systolic norm of Reinhardt domains, also known as toric domains. Let $\R^4$ be equipped with coordinates $(x_1,y_1,x_2,y_2)$ and standard symplectic form $\omega_0 = \sum_j r_j dr_j \wedge d\theta_j$ written in polar coordinates $x_j=r_j\cos\theta_j$, $y_j=r_j\sin\theta_j$. Let $F:\R^2\to[0,+\infty)$ be a $1$-homogeneous function, smooth and positive away from the origin. If we set $H = F(r_1^2,r_2^2)$ then $W = \{ H \leq 1\}$ is a smooth domain and the Liouville form $\lambda_0 = \frac{1}{2} \sum_j r_j^2d\theta_j$ defines a contact form on $\partial W \simeq S^3$. The proof of the next proposition is given in section~\ref{sec_toric}.

\begin{proposition}\label{prop_toric}
Among boundaries of Reinhardt domains $\rho_- \leq 1 \leq \rho_+$ always holds and, moreover, the systolic norm vanishes precisely for ellipsoids.
\end{proposition}


\noindent {\sc Question 1.} {\it What kinds of Reeb dynamics are allowed for contact forms $\lambda$ on~$S^3$ satisfying $\|\lambda\|_{\rm sys}=0$?} 

\bigskip

\noindent {\sc Question 2.} {\it Is it true that for a Reeb flow on~$S^3$ with only two periodic orbits the contact volume coincides with the product of the periods?}

\bigskip

\noindent {\sc Question 3.} {\it Is the set of values $\rho(\gamma,\gamma')$, when $(\gamma,\gamma')$ varies over pairs of geometrically distinct periodic orbits, dense in the systolic interval?}

\bigskip

Later we will transport the above definitions to an arbitrary $(M,\lambda)$, at which point we will be in position to explain why inequalities such as $1\leq\sup_\gamma\rho(\gamma_0,\gamma)$, with $\gamma_0$ fixed {\it a priori}, are precisely versions for Reeb flows on $S^3$ of the results from~\cite{mean,weiler}. We are then led to ask the following questions. 

\bigskip

\noindent {\sc Question 4.} {\it Does $1\in I(\lambda)$ hold for every contact form $\lambda$ on the $3$-sphere?} 

\bigskip

\noindent {\sc Question 5.} {\it More generally, do the systolic 
inequalities $$ \inf_{\gamma \subset S^3\setminus\gamma_0} \rho(\gamma_0,\gamma) \leq 1 \leq \sup_{\gamma \subset S^3\setminus\gamma_0} \rho(\gamma_0,\gamma) $$ hold
for every periodic Reeb orbit $\gamma_0$ of an arbitrary contact form on $S^3$?}

\bigskip

\begin{remark}\label{rmk_H}
Hutchings suggested to enlarge the systolic interval by considering also the values of the quadratic linking form at certain points of the diagonal, namely $\rho(\gamma,\gamma) = {\rm rot}_0(\gamma)\vol(\lambda)T(\gamma)^{-2}$. Here ${\rm rot}_0$ denotes the transverse rotation number computed in a Seifert framing. Let us refer to this interval as the {\it enlarged systolic interval}. Hutchings then explained that the result from~\cite{asymptotic_ech} can be combined with arguments similar to those from~\cite{hutchings_ruelle} to conclude that~$1$ belongs to the enlarged systolic interval. We believe it to be an interesting question to decide if the enlarged systolic interval always coincides with the systolic interval. This is true for Reinhardt domains; see Remark~\ref{rmk_rot_numbers}.
\end{remark}

\begin{remark}
Note that $\rho_\pm$ induce numerical invariants of symplectic fillings of the tight three-sphere. In fact, let $(W,\omega)$ be a symplectic $4$-manifold with boundary, and denote by $\iota:\partial W \to W$ the inclusion map. We orient $W$ by $\omega \wedge \omega$, and $\partial W$ gets the boundary orientation. Assume that there is a $1$-form $\alpha$ defined near $\partial W$ and an orientation preserving diffeomorphism $\phi : S^3 \to \partial W$ such that $d\alpha=\omega$, and $\lambda:=(\iota\circ\phi)^*\alpha$ satisfies $\xi_\mathrm{std} = \ker\lambda$. Here $\xi_\mathrm{std}$ denotes the standard contact structure on $S^3$, and $S^3$ is oriented by $\xi_{\rm std}$. Then we can define $\rho_\pm(W,\omega)=\rho_\pm(\lambda)$. It turns out that the numbers $\rho_\pm(W,\omega)$ do not depend on $\alpha$ and $\phi$, and if $(W_1,\omega_1)$ is symplectomorphic to $(W_2,\omega_2)$ then $\rho_\pm(W_1,\omega_1) = \rho_\pm(W_2,\omega_2)$.
\end{remark}

\begin{remark}\label{rmk_systolic_4}
Invariants $s(\lambda),S(\lambda)$ of a contact form $\lambda$ on the tight $3$-sphere were introduced in~\cite[page 2647]{compositio}. These numbers are defined in terms of rotation numbers of periodic orbits and the minimal action $T_{\rm min}(\lambda)$. In~\cite[Theorem~1.2]{compositio} one finds an existence statement of contact forms with systolic ratio $T_{\rm min}(\lambda)^2/\vol(\lambda)$ close to $n$, such that the length $\Delta(\lambda) = S(\lambda)-s(\lambda)$ of the interval $[s(\lambda),S(\lambda)]$ is close to~$n^2$. It is expected that this behaviour is sharp for the relation between the systolic ratio and $\Delta(\lambda)$. In particular, it is expected that $\Delta(\lambda)$ must always diverge together with the systolic ratio. The case $n=2$ provides dynamically convex examples with systolic ratio close to $2$.
\end{remark}

Let $(M,\lambda)$ be a contact-type energy level. A finite collection $\alpha = \{\gamma_j\}$ of periodic Reeb orbits is called an orbit set\footnote{This terminology slightly conflicts with the one used in ECH.}. Each $\gamma_j$ can be also seen as a map $\gamma_j:\R/T(\gamma_j)\Z\to M$, where $T(\gamma_j)$ is the (not necessarily primitive) period of $\gamma_j$. Let $\mu$ be an invariant Borel probability measure. We shall say that $\mu$ can be approximated by periodic orbits if there exists a sequence of orbit sets $\alpha_n=\{\gamma_j^n\}$ and a sequence of positive weights $\{p^n_j\}$ satisfying $\sum_jp^n_j=1 \ \forall n$ and $$ \lim_{n\to\infty} \sum_j p^n_j \frac{(\gamma_j^n)_*\leb}{T(\gamma_j^n)} = \mu $$ as measures. The Liouville measure is the measure induced by $\lambda\wedge d\lambda$, and the normalised Liouville measure $\mu_\lambda$ is the one induced by $\frac{\lambda\wedge d\lambda}{\vol(\lambda)}$. We shall briefly say that the Liouville measure is approximated by periodic orbits if so is $\mu_\lambda$ in the sense explained above.

\begin{theorem}\label{thm_eq_1}
If the Liouville measure on $(S^3,\lambda)$ can be approximated by periodic orbits, then for every closed Reeb orbit $\gamma_0$ and every $\epsilon>0$ the sets
\begin{equation}
\overline{\bigcup_{\substack{\gamma \subset S^3 \setminus \gamma_0 \\ \rho(\gamma,\gamma_0)\geq1-\epsilon}} \gamma} \qquad \qquad \overline{\bigcup_{\substack{\gamma \subset S^3 \setminus \gamma_0 \\ \rho(\gamma,\gamma_0)\leq1+\epsilon}} \gamma}
\end{equation}
have positive Liouville measure. In particular, $1\in I(\lambda)$ and
\begin{equation}\label{syst_ineq_3sphere}
\inf_{\gamma\subset S^3\setminus\gamma_0} \rho(\gamma,\gamma_0) \leq 1 \leq \sup_{\gamma\subset S^3\setminus\gamma_0} \rho(\gamma,\gamma_0)
\end{equation}
holds for every closed Reeb orbit $\gamma_0$.
\end{theorem}

Theorem~\ref{thm_eq_1} provides evidence for a positive answer to questions~4 and~5. The reason is~\cite[Corollary~1.4]{Irie_eq} asserting that for a $C^\infty$-generic Reeb flow the Liouville measure can be approximated by periodic orbits. Hence, $C^\infty$-generically the conclusions of Theorem~\ref{thm_eq_1} hold, and questions~4 and~5 have a positive answer.

The inequality $1\leq\rho_+(\lambda)$ is connected to a conjecture of Viterbo~\cite{viterbo}. A particular case of the conjecture asks for the inequality
\begin{equation}\label{viterbo_ineq}
1\leq \frac{\vol(\lambda)}{T_\mathrm{min}(\lambda)^2}
\end{equation}
to hold whenever the contact form on $S^3$ is induced by a convex energy level in a symplectic $4$-dimensional vector space. If there is a positive lower bound for the systolic pairings $\rho(\gamma_1,\gamma_2)$, as $\gamma_1,\gamma_2$ vary on special subsets of periodic orbits with a control on the linking number, then a systolic inequality follows. Moreover, equality in~\eqref{viterbo_ineq} should hold exactly when the contact form is Zoll, i.e. all Reeb trajectories are periodic and have the same primitive period. This case of the conjecture was verified in~\cite{abhs2} for convex sets on a $C^3$-neighbourhood of the round $3$-sphere in $\R^4$ with its standard symplectic form, but the proof is indirect and does not provide more information than~\eqref{viterbo_ineq}. In~\cite{BK} the correct analogue of the inequality~\eqref{viterbo_ineq} in more general $3$-manifolds was formulated, and then proved on a $C^3$-neighbourhood of Zoll contact forms. Recently a version of this result in any dimension appeared in~\cite{AB}.

Let us generalise the above discussion to all $3$-manifolds. Fix a contact-type energy level $(M,\lambda)$. An adapted Seifert surface in $(M,\lambda)$ is defined here to be a smoothly embedded\footnote{We assume clean intersections with $\partial M$.}, connected, orientable compact surface $\Sigma\hookrightarrow M$ such that $\partial \Sigma \setminus \partial M$ consists of periodic Reeb orbits, or is empty. We split the boundary of $\Sigma$ according to
\begin{equation}
\partial \Sigma = \dot\partial \Sigma \sqcup \partial_b\Sigma \qquad \text{where} \qquad \dot\partial \Sigma = \partial \Sigma \setminus\partial M, \ \partial_b\Sigma = \partial\Sigma \cap \partial M
\end{equation}
Note that there may be no choice of orientation for $\Sigma$ that orients all components of $\dot\partial\Sigma$ along the flow. Once an orientation of $\Sigma$ is fixed, its contact area is $$ T(\Sigma) = \int_\Sigma d\lambda. $$

\begin{definition}
If $\Sigma$ is an adapted Seifert surface in $(M,\lambda)$ and $\gamma$ is a periodic Reeb orbit in $M\setminus\dot\partial\Sigma$ then we define a \textit{systolic pairing}
\begin{equation}
\rho(\gamma,\Sigma) = \frac{\inter(\gamma,\Sigma)\vol(\lambda)}{T(\gamma)T(\Sigma)}
\end{equation}
provided $T(\Sigma)\neq0$.
\end{definition}

The number $\rho(\gamma,\Sigma)$ does not depend on the orientation of $\Sigma$. It is invariant under iterations of $\gamma$ and rescaling of $\lambda$.
Theorem~\ref{thm_eq_1} is a direct consequence of the next statement.

\begin{theorem}\label{thm_eq_2}
Let $\Sigma$ be an oriented, adapted Seifert surface in $(M,\lambda)$ such that $T(\Sigma) \neq 0$. If the Liouville measure can be approximated by periodic orbits then for every $\epsilon>0$ the sets
\begin{equation}\label{sets_of_per_orbits}
\overline{\bigcup_{\substack{\gamma \subset M \setminus \dot\partial\Sigma \\ \rho(\gamma,\Sigma)\geq1-\epsilon}} \gamma} 
\qquad \qquad \overline{\bigcup_{\substack{\gamma \subset M \setminus \dot\partial\Sigma \\ \rho(\gamma,\Sigma)\leq1+\epsilon}} \gamma}
\end{equation}
have positive Liouville measure. In particular, 
\begin{equation}\label{syst_ineq}
\inf_{\gamma\subset M\setminus\dot\partial\Sigma} \rho(\gamma,\Sigma) \leq 1 \leq \sup_{\gamma\subset M\setminus\dot\partial\Sigma} \rho(\gamma,\Sigma)
\end{equation}
holds.
\end{theorem}

The proof relies on the Action-Linking Lemma (Lemma~\ref{lem_action_linking}), and will be given in section~\ref{sec_proof_thm}. Our main application reads as follows.

\begin{theorem}\label{thm_anosov}
If the Reeb flow of $\lambda$ is Anosov, then for every adapted Seifert surface $\Sigma\subset M$ such that $T(\Sigma)\neq0$, and every $\epsilon>0$, the sets~\eqref{sets_of_per_orbits} have positive Liouville measure, and~\eqref{syst_ineq} holds.
\end{theorem}

\begin{proof}
By the result of Sigmund~\cite{sigmund}, any invariant measure on a basic piece of the non-wandering set of an Axiom A flow can be approximated by measures which are given by one periodic orbit. Hence the same is true for any invariant Borel probability measure of an Anosov Reeb flow on a connected compact $3$-manifold. In particular, this is true for the normalised Liouville measure on each connected component of $M$. Now apply Theorem~\ref{thm_eq_2}.
\end{proof}

\begin{remark}\label{rmk_Irie_eq}
Let $M$ be a smooth compact $3$-manifold. For $1\leq k\leq\infty$ let $\mathscr{C}_k(M)$ be the space of contact forms on $M$ of class $C^k$ such that the Reeb vector field is tangent to $\partial M$. Equip $\mathscr{C}_k(M)$ with the $C^k$-topology. This topology can be defined by a metric with respect to which every point has a complete neighbourhood. The following statement is a direct application of Theorem~\ref{thm_eq_2} combined with~\cite[Corollary~1.4]{Irie_eq}: {\it There exists a residual subset $\mathscr{R}$ of $\mathscr{C}_\infty(M)$ such that the sets~\eqref{sets_of_per_orbits} have positive Liouville measure, and~\eqref{syst_ineq} holds for every $\lambda \in \mathscr{R}$ and every adapted Seifert surface in $(M,\lambda)$ with non-zero contact area. In particular,~\eqref{syst_ineq_3sphere} holds for a $C^\infty$-generic contact form on $S^3$.}
\end{remark}

The proof of Theorem~\ref{thm_eq_2} uses the Action-Linking Lemma, which we consider of independent interest. Consider a general vector field on a compact $3$-manifold $M$ that is tangent to $\partial M$. Let $L$ be a finite collection of non-constant periodic orbits contained in $M\setminus \partial M$. The theory of asymptotic cycles provides a way to define a real-valued intersection number between an invariant Borel probability measure $\mu$ on $M\setminus L$ and a cohomology class $y \in H^1(M\setminus L;\R)$. This is defined as follows. For a recurrent point $p\in M\setminus L$ we can consider sequences $t_n\to+\infty$ such that $\phi^{t_n}(p) \to p$ and study the limits $\left<y,k(t_n,p)\right>/t_n$ as $n\to \infty$, where the loop $k(t_n,p)$ is obtained by concatenating to $\phi^{[0,t_n]}(p)$ a short path from $\phi^{t_n}(p)$ to $p$. The classical ergodic theorem implies that for $\mu$-almost all recurrent points these limits exist independently of the choice of sequence $t_n$ and of the closing short paths, and define a $\mu$-integrable function $f_{\mu,y}$; its values are determined on an invariant set of full measure with respect to $\mu$, which can be chosen to contain only recurrent points. In the complement of such a set we can define $f_{\mu,y}$ to be zero, or choose any other measurable extension. The intersection number is defined as 
\begin{equation}\label{function_int_number}
\mu \cdot y = \int_{M\setminus L} f_{\mu,y} \ d\mu
\end{equation}
For the action-linking lemma we consider the Reeb vector field defined by $\lambda$, set $L = \dot\partial\Sigma$, and take $y = \Sigma^*$ the class ``dual'' to $\Sigma$. This means that $\left<\Sigma^*,c\right> = \inter(c,\Sigma)$ holds for every $1$-cycle in $M\setminus\dot\partial\Sigma$. Finally we define $$ \inter(\lambda\wedge d\lambda,\Sigma) = \vol(\lambda) \ \left(\frac{\lambda\wedge d\lambda}{\vol(\lambda)}\right) \cdot \Sigma^* $$ where $\lambda\wedge d\lambda$ is viewed as an invariant measure.

\begin{lemma}[Action-linking lemma]\label{lem_action_linking}
The identity $$ \inter(\lambda\wedge d\lambda,\Sigma)=T(\Sigma) $$ holds for every $(M,\lambda)$ and every oriented adapted Seifert surface $\Sigma\subset M$.
\end{lemma}

The proof of Lemma~\ref{lem_action_linking} is found in Section~\ref{sec_proof_action_link}. For an introductory exposition on the theory of asymptotic cycles we refer to~\cite{SFS}.

\begin{remark}
The Action-Linking Lemma can be seen as a version of the main result from~\cite{david} for situations where an appropriate global surface of section may not be available. However, differently from~\cite{david}, we do not handle non-periodic orbits.
\end{remark}

An elementary application is the existence of {\it wise travelers} when the Liouville measure is ergodic. The definition of wise travelers is explained in the following statement.

\begin{lemma}[Wise traveler lemma]\label{lem_wise}
Let $(M,\lambda)$ be a contact-type energy level. Assume that the Liouville measure is ergodic, and let $\Sigma_j$ be any countable collection of adapted Seifert surfaces. Then there exists a Borel subset $E\subset M\setminus \cup_j\dot\partial\Sigma_j$ of full Liouville measure with the following property: The trajectory of any point of $E$ contains the information of all numbers $T(\Sigma_j)$, in the sense that $E$ consists of recurrent points, and if $p\in E$ then for every $j$ and every sequence $t_n\to+\infty$ such that $\phi^{t_n}(p) \to p$ the following holds: $$ T(\Sigma_j) = \vol(\lambda) \lim_{n\to+\infty} \frac{1}{t_n} \ \inter(k(t_n,p),\Sigma_j). $$
\end{lemma}

On a homology $3$-sphere it follows from Lemma~\ref{lem_wise} that if the Liouville measure is ergodic then for every countable collection of periodic orbits $\gamma_j$ we can find a wise traveler whose trajectory contains the information of all $T(\gamma_j)$. For non-degenerate contact forms the set $\{\gamma_j\}$ can be all the periodic orbits.

The following is a consequence of the fact, going back to Anosov's thesis, that the Liouville measure of an Anosov Reeb flow defined on a connected $3$-manifold is ergodic.

\begin{corollary}
Wise travellers always exist for an Anosov Reeb flow on a connected $3$-manifold.
\end{corollary}

\begin{proof}[Proof of Lemma~\ref{lem_wise}]
Let $X$ be the Reeb vector field of $\lambda$. Denote by $\phi^t$ its flow. It is clear from its definition that $f_{\mu_\lambda,\Sigma_j^*}$ is invariant by the flow. By the ergodicity assumption we find a Borel set $Y_j \subset M\setminus \dot\partial\Sigma_j$ of full Liouville measure, contained in the set of recurrent points, such that  $$ \inter(\lambda \wedge d\lambda,\Sigma_j) = \vol(\lambda) \ f_{\mu_\lambda,\Sigma_j^*}(q) \qquad \forall q\in Y_j. $$ Since $E := \cap_j Y_j$ is still of full Liouville measure, and since by the action-linking lemma $T(\Sigma_j) = \inter(\lambda\wedge d\lambda,\Sigma_j)$, we can choose our wise traveler to be any point in~$E$.
\end{proof}

\subsection{Discussion about results on mean action}\label{ssec_HW}

Let $\D \subset \C$ be the unit disk with coordinates $z=x+iy$ and area form $\omega_0=dx\wedge dy$. Let $H:\R/\Z\times\D\to\R$ be a Hamiltonian and denote $H_t=H(t,\cdot)$. The Hamiltonian vector field $X_{H_t}$ defined by $dH_t=\omega_0(X_{H_t},\cdot)$ is assumed to be tangent to $\partial \D$, for every $t$. Denote by $\varphi^t_H$ the isotopy obtained by integrating $X_{H_t}$ starting from the identity at time zero, and denote $h=\varphi^1_H$. The isotopy defines a rotation number $\rho_H \in \R$ of $h|_{\partial\D}$.
Given a primitive $\eta$ of $\omega_0$ the associated action function is
\begin{equation}\label{action_function}
\sigma_{H,\eta}(z) = \int_{\varphi^{[0,1]}_H(z)} \eta + \int_0^1 H_t(\varphi^t_H(z))dt
\end{equation}
The Calabi invariant of $H$ is
\begin{equation}
\cal(H) = \frac{1}{\pi} \int_\D\sigma_{H,\eta} \ \omega_0
\end{equation}
which turns out to be independent of $\eta$. If $z$ is a $k$-periodic point of $h$ then its action and its mean action are defined as 
\begin{equation}\label{mean_actions}
\sigma_H(z,k) := \sum_{i=0}^{k-1} \sigma_{H,\eta}(h^i(z)) \qquad\qquad \bar \sigma_H(z) := \frac{\sigma_H(z,k)}{k}
\end{equation}
respectively. Note that $\sigma_H(z,k)$ is independent of $\eta$ and that $\bar \sigma_H(z)$ does not depend on the choice of period. Action, mean action and Calabi invariant here are equal to $\pi$ times the corresponding quantities from~\cite{mean}.

\begin{theorem}[Hutchings~\cite{mean}]\label{thm_mean}
Assume that $h(re^{i\theta})=re^{i(\theta+\rho_H)}$ near $\partial\D$, that $H_t|_{\partial\D}=0$ for all $t$ and that $\eta$ agrees with $(1/2)d\theta$ on $\partial\D$. If $\cal(H) < \frac{\rho_H}{2}$ then for every $\epsilon>0$ there is a periodic point $z$ such that $\bar \sigma_H(z) \leq \cal(H)+\epsilon$.
\end{theorem}

Weiler's theorem~\cite{weiler} is the analogue for maps on the annulus. It is trivial, but important, to note that if some sequence of finite collections of periodic points of $h$ becomes equidistributed with respect to $\omega_0/\pi$ then for every $\epsilon>0$ there is a periodic point~$z$ such that $\bar\sigma_H(z) \leq \cal(H)+\epsilon$, i.e. the conclusion of Theorem~\ref{thm_mean} holds with no extra assumptions. This remark can be translated to the language of contact forms and Reeb vector fields as follows. If $c$ is large then $\lambda_c = (H+c)dt+\eta$ is a contact form on $M=\R/\Z\times\D$ whose Reeb vector field $X_{\lambda_c}$ is a positive multiple of $\tilde X_H = \partial_t+X_{H_t}$. Fix such $c$. For every $z\in\D$ the integral of $\lambda_c$ along a piece of $\tilde X_H$-trajectory of the form $t\in[0,1]\mapsto (t,\varphi^t_H(z))\in M$ is equal to $\sigma_{H,\eta}(z)+c$. Hence the action (period) of a closed orbit $\gamma$ of $X_{\lambda_c}$ through a point $(0,z)$ is equal to $T(\gamma) = \int_\gamma \lambda_c = \sigma_H(z,k)+kc$ where $k = \inter(\gamma,\{0\}\times \D)$. Note also that $\vol(\lambda_c) = \pi(\cal(H) + c)$. If we assume that some sequence of orbit sets for $\lambda_c$ becomes Liouville equidistributed then we can apply Theorem~\ref{thm_eq_2} with $\Sigma = \{0\}\times\D$ to find, for every $\epsilon>0$, a periodic Reeb orbit such that
\begin{equation}
\rho(\gamma,\Sigma) \geq 1-\epsilon \qquad \Leftrightarrow \qquad \frac{1}{1-\epsilon} \cal(H) + \frac{\epsilon c}{1-\epsilon} \geq \bar\sigma_H(z)
\end{equation}
The conclusion follows from letting $\epsilon\to0$. Under the same hypothesis one can argue analogously and obtain, for any $\epsilon>0$, a periodic point such that $\bar\sigma_H \geq \cal(H) - \epsilon$. The reason why this application of Theorem~\ref{thm_eq_2} is trivial is the obvious equality between the period of a periodic point of $h$ and the intersection number between the corresponding periodic Reeb orbit of $\lambda_c$ and $\Sigma$. This follows from the fact that $\Sigma$ is a global surface of section. The role of the Action-Linking Lemma is to handle situations where a global surface of section with the appropriate boundary behaviour might not be available. \\

\noindent {\bf Acknowledgments.} We thank Marcelo Alves and Gerhard Knieper for helpful discussions, especially for pointing us to Sigmund's work. We thank Michael Hutchings for comments on an earlier version of this paper leading to Remark~\ref{rmk_H}. We thank Barney Bramham for insightful discussions. We thank the referees for pointing out corrections, for the helpful feedback, and for suggesting Question~3. U.~Hryniewicz thanks the participants of the {\it Semin\'ario sobre Campos de Beltrami} held in 2018 at UFRJ (Rio de Janeiro) for fruitful discussions. D. Bechara Senior and U.~Hryniewicz acknowledge support by the DFG SFB/TRR 191 `Symplectic Structures in Geometry, Algebra and Dynamics', Projektnummer 281071066-TRR 191. PS acknowledges the support of NYU-ECNU Institute of Mathematical Sciences at NYU Shanghai. PS is partially supported by FAPESP 2016/25053-8 and CNPq 306106/2016-7.

\section{Proof of the Action-Linking Lemma}\label{sec_proof_action_link}

The first step is to blow the periodic orbits in $\dot\partial\Sigma$ up. We get a new compact smooth $3$-manifold $\hat M$ with boundary, it is obtained by adding one boundary torus at each end of $M\setminus \dot\partial\Sigma$. The closure $\hat\Sigma$ of $\Sigma \setminus \dot\partial\Sigma$ in $\hat M$ is an embedded surface with clean intersections with the boundary. Moreover, the vector field $X$ extends smoothly to a vector field $\hat X$ on $\hat M$ that is tangent $\partial\hat M$. We refer to~\cite[Section 3]{SFS}. For the sake of completeness we provide details.

Order the periodic orbits in $\dot\partial\Sigma$ as $c_1,c_2\dots$, denote by $T_j$ the primitive period of~$c_j$. Choose an orientation preserving diffeomorphism $\Psi_j:\mathbb{R}/T_j\Z\times\D \to N_j$ onto a small compact tubular neighborhood $N_j$ of $c_j$ such that
\[
\Psi_j(t,0)=c_j(t) \qquad \text{and} \qquad \Psi_j^{-1}(N_j \cap \Sigma) = \R/T_j\Z \times [0,1] \times \{0\}
\]
On $N_j\setminus c_j$ we get tubular polar coordinates
\[
N_j\setminus c_j \simeq \R/T_j\Z \times (0,1] \times \R/2\pi\Z \qquad p = \Psi_j(t,r\cos\theta,r\sin\theta) \simeq (t,r,\theta)
\]
Define
\[
\hat M = \left. \left\{ (M\setminus\dot\partial\Sigma) \ \sqcup \ \bigsqcup_j \R/T_j\Z \times [0,1] \times \R/2\pi\Z \right\} \right/ \sim
\]
where a point in $N_j\setminus c_j$ is identified with the point in $\R/T_j\Z \times (0,1] \times \R/2\pi\Z$ given by its tubular polar coordinates. The obvious differentiable structure on $\hat M$ turns it into a compact smooth $3$-manifold with a new boundary torus $$ \T_j =  \R/T_j\Z \times \{0\} \times \R/2\pi\Z $$ for each $j$. Moreover, the closure $\hat \Sigma$ of $\Sigma \setminus \dot\partial\Sigma$ in $\hat M$ intersects $\partial\hat M$ cleanly.

Note that $\partial_\theta$ is transverse to $\Sigma$ in $N_j\setminus c_j$. Define $\epsilon_j = +1$ if $\partial_\theta$ is positively transverse, or $\epsilon_j=-1$ otherwise. We fix a choice of smooth vector field $Y$ on $\hat M$ that is positively transverse to $\hat\Sigma$, coincides with $\epsilon_j\partial_\theta$ on $\R/T_j\Z \times [0,1] \times \R/2\pi\Z$, and is also tangent to all other boundary components of $\hat M$. Using the flow of $Y$ we define a smooth diffeomorphism
\begin{equation}
F:\U \to \hat\Sigma\times [-\delta_0,\delta_0]
\end{equation}
such that $F_*Y=\partial_z$ where $\U$ is a compact neighborhood of $\hat \Sigma$ in $\hat M$, $\delta_0>0$ is small, and $z$ is the coordinate on $[-\delta_0,\delta_0]$. For every $\delta \in (0,\delta_0)$ consider a smooth function $\varphi_\delta:\R\to[0,+\infty)$ satisfying $$ \supp(\varphi_\delta) \subset [-\delta,\delta] \qquad \int_\R \varphi_\delta(z)dz=1 $$ Let $\beta_\delta$ be the $1$-form defined as $F^*(\varphi_\delta(z)dz)$ on $\U$, and $0$ on $\hat M\setminus \U$. It follows that $\beta_\delta$ is smooth on $\hat M$, $d\beta_\delta=0$ and $\beta_\delta$ and represents the dual class $\Sigma^*$ on $M\setminus\dot\partial\Sigma\subset \hat M$. Fix a positive area form $\Omega$ on $\hat\Sigma$. Define a function
\begin{equation}\label{function_g}
g: \U \cap (M\setminus\dot\partial\Sigma) = \U \setminus \cup_j\T_j \to (0,+\infty) \qquad \text{by} \qquad \lambda \wedge d\lambda = g \ \Omega\wedge dz
\end{equation}
where $\U$ gets identified with $\hat\Sigma \times [-\delta_0,\delta_0]$ via $F$. \\

\noindent {\bf Claim.} For every $j$ we have $\sup_{|z|\leq\delta_0} g(q,z) \to 0$ as $q\to \partial\hat\Sigma \cap \T_j$. \\

\noindent {\it Proof of Claim.} $\Psi_j$ induces coordinates $(t,x,y) \in \R/T_j\Z\times\D$ on $N_j$ and tubular polar coordinates $(t,r,\theta) \in \R/T_j\Z\times [0,1]\times\R/2\pi\Z$ near $\T_j$. We can write
\[
\lambda = adt+bdx+cdy \qquad d\lambda = A\ dx\wedge dy + B\ dy\wedge dt + C\ dt\wedge dx
\]
on $N_j$. Note that $a,b,c,A,B,C$ are smooth in $(t,x,y)$, hence they are also smooth in $(t,r,\theta)$ all the way up to $\{r=0\}$. We get an expression
\[
\lambda \wedge d\lambda = (aA+bB+cC) r \ dt\wedge dr\wedge d\theta
\]
Note that $(t,r)$ are smooth coordinates in $\hat \Sigma$ near $\partial\hat\Sigma \cap \T_j$. Hence we can write $\Omega = h(t,r) dt\wedge dr$ near $\partial\hat\Sigma \cap \T_j$ for some $h$ that is smooth up to $\{r=0\}$. The sign of $h$ depends on whether $dt\wedge dr$ is positive or negative on $\hat\Sigma$. Moreover, since $Y$ coincides with $\epsilon_j\partial_\theta$ near the $\T_j$, we get $$ \Omega \wedge dz = \epsilon_j h(t,r) dt\wedge dr\wedge d\theta $$ near $\T_j$, with a positive coefficient $\epsilon_j h(t,r)$ in front of $dt\wedge dr\wedge d\theta$. We finally get an expression for $g$:
\[
g = \frac{(aA+bB+cC)r}{\epsilon_j h(t,r)}
\]
which is $O(r)$ as $r\to0$ since the denominator is bounded away from zero. \qed \\

Let $\Phi_j:\R/T_j\Z\times[0,1]\times\R/2\pi\Z \to M$ be the map $\Phi_j(t,r,\theta) = \Psi_j(t,r\cos\theta,r\sin\theta)$. It is smooth and  defines a diffeomorphism $\R/T_j\Z\times(0,1]\times\R/2\pi\Z \simeq N_j\setminus c_j$ that we use to pull $X$ back to a vector field $W_j$ defined on $\R/T_j\Z\times(0,1]\times\R/2\pi\Z$. In~\cite[Section 3]{SFS} it is proved that $W_j$ admits a smooth extension to $\R/T_j\Z\times[0,1]\times\R/2\pi\Z$ that is tangent to the boundary torus $\T_j$. It follows that the vector field $X$ restricted to $M\setminus\dot\partial\Sigma$ extends smoothly to a vector field $\hat X$ on $\hat M$ tangent to~$\partial\hat M$.

Both $\hat X$ and $\beta_\delta$ are smooth objects defined on the compact manifold $\hat M$. Hence $i_{\hat X}\beta_\delta$ is bounded. Since $X$ and $\hat X$ coincide on $M \setminus \dot\partial\Sigma = \hat M \setminus \partial\hat M$, we get that $i_{X}\beta_\delta$ is bounded on $M \setminus \dot\partial\Sigma$.
It turns out that
\begin{equation}\label{inter_integral}
\inter(\lambda\wedge d\lambda,\Sigma) = \int_{M\setminus \dot\partial \Sigma} i_X\beta_\delta \ \lambda \wedge d\lambda \qquad \forall \delta \in (0,\delta_0)
\end{equation}
This is a simple application of Birkhoff's ergodic theorem; details are spelled out in~\cite[subsection~2.1]{SFS}. In particular, the integrals on the right hand side of~\eqref{inter_integral} do not depend on $\delta$. The rest of this proof consists in showing that the limit as $\delta\to0$ in the above identity is equal to $T(\Sigma)$.

Denote by $f:\U \to \R$ the smooth function $i_{\hat X}dz$. Since $\U$ is compact, its $L^\infty$-norm $\|f\|_{L^\infty(\U)}$ is finite. Denote $\hat\Sigma_0 := \hat \Sigma \setminus \cup_j\T_j = \Sigma\setminus \dot\partial\Sigma$.

Let $\epsilon>0$ be arbitrary. Choose a smooth compact domain $K \subset \hat\Sigma_0$ such that
\begin{equation}\label{set_K1}
\left| \int_{\hat\Sigma \setminus K} \Omega \right| \leq 1 \qquad\qquad \|g\|_{L^\infty(\hat\Sigma_0\setminus K \times [-\delta_0,\delta_0])}\|f\|_{L^\infty(\U)} \leq \epsilon
\end{equation}
and that
\begin{equation}\label{set_K2}
\left| \int_{\hat\Sigma_0\setminus K} f(q,0)g(q,0) \ \Omega \right| \leq \epsilon
\end{equation}
Here we used the {\it Claim} proved above. We have
\begin{equation}
\begin{aligned}
\int_{M\setminus\dot\partial\Sigma} i_{\hat X}\beta_\delta \ \lambda \wedge d\lambda &= \int_{\hat\Sigma_0 \times [-\delta,\delta]} \varphi_\delta(z) f(q,z)g(q,z) \ \Omega\wedge dz \\
&= \int_{\hat\Sigma_0 \setminus K \times [-\delta,\delta]} \varphi_\delta(z) f(q,z)g(q,z) \ \Omega\wedge dz \\
&+ \int_{K \times [-\delta,\delta]} \varphi_\delta(z) f(q,z)g(q,z) \ \Omega\wedge dz \\
&=: {\rm (I)} + {\rm (II)}
\end{aligned}
\end{equation}
where $g$ is the function~\eqref{function_g}. Using~\eqref{set_K1} we estimate
\begin{equation}\label{estimate_I}
|{\rm (I)}| \leq \epsilon \int_\R \varphi_\delta(z)dz = \epsilon
\end{equation}
To estimate (II) we write $fg$ on $K \times [-\delta_0,\delta_0]$ as
\begin{equation}
f(q,z)g(q,z) = f(q,0)g(q,0) + \Delta(q,z)
\end{equation}
for some function $\Delta:K \times [-\delta_0,\delta_0]\to\R$ satisfying
\begin{equation}
|\Delta(q,z)| \leq C|z| \qquad C = \|d(fg)\|_{L^\infty(K \times [-\delta_0,\delta_0])}
\end{equation}
It is important to note that $C$ depends on $K$, hence also on $\epsilon$, but not on $\delta$. Then we can estimate
\begin{equation}\label{estimate_II}
\begin{aligned}
& \left| {\rm (II)} - \int_{K \times [-\delta,\delta]} \varphi_\delta(z)f(q,0)g(q,0) \ \Omega\wedge dz \right| \\
&\leq C \delta \int _{K \times [-\delta,\delta]} \varphi_\delta(z) \  \Omega\wedge dz = C\delta \int_{K} \Omega \leq C\delta \int_{\hat\Sigma} \Omega
\end{aligned}
\end{equation}
We can now finally estimate using~\eqref{set_K2},~\eqref{estimate_I} and~\eqref{estimate_II}
\begin{equation*}
\begin{aligned}
& \left| \inter(\lambda\wedge d\lambda,\Sigma) - \int_{\hat\Sigma_0} f(q,0)g(q,0) \ \Omega \right| \\
& = \left| \int_{M\setminus\dot\partial\Sigma} i_X\beta_\delta \ \lambda \wedge d\lambda - \int_{\hat\Sigma_0} f(q,0)g(q,0) \ \Omega \right| \\
&\leq |{\rm (I)}| + \left| {\rm (II)} - \int_{K \times [-\delta,\delta]} \varphi_\delta(z)f(q,0)g(q,0) \ \Omega\wedge dz \right| \\
&+ \left|\int_{K \times [-\delta,\delta]} \varphi_\delta(z)f(q,0)g(q,0) \ \Omega\wedge dz - \int_{\hat\Sigma_0} f(q,0)g(q,0) \ \Omega  \right| \\
&\leq \epsilon + C\delta \int_{\hat\Sigma} \Omega + \left| \int_{\hat\Sigma_0\setminus K} f(q,0)g(q,0) \ \Omega \right|  \\
&\leq 2\epsilon + C\delta \int_{\hat\Sigma} \Omega
\end{aligned}
\end{equation*}
Letting $\delta\to0$ above we get
\[
\left| 
\inter(\lambda\wedge d\lambda,\Sigma) - \int_{\hat\Sigma_0} f(q,0)g(q,0) \ \Omega \right| \leq 2\epsilon
\]
and since $\epsilon>0$ is arbitrary we conclude that
\begin{equation}\label{almost_final}
\inter(\lambda\wedge d\lambda,\Sigma) = \int_{\hat\Sigma_0} f(q,0)g(q,0) \ \Omega
\end{equation}
The final step is to compute the integral on the right hand side of~\eqref{almost_final}. On $\U\setminus \cup_j\T_j$ we compute
\begin{equation}
d\lambda = i_X(\lambda \wedge d\lambda) = i_X(g \ \Omega\wedge dz) = fg \ \Omega + dz \wedge \nu
\end{equation}
for some $1$-form $\nu$. Since $dz$ vanishes tangentially to $\hat\Sigma$ we get 
\[
T(\Sigma) = \int_\Sigma d\lambda = \int_{(\Sigma\setminus\dot\partial\Sigma) = \hat\Sigma_0} d\lambda = \int_{\hat\Sigma_0} f(q,0)g(q,0) \ \Omega
\]
These identities and~\eqref{almost_final} imply that
\[
\inter(\lambda\wedge d\lambda,\Sigma) = T(\Sigma)
\]
The proof of the Action-Linking Lemma is complete.

\section{Proof of Theorem~\ref{thm_eq_2}}\label{sec_proof_thm}

Orient $\Sigma$ so that $T(\Sigma)>0$. Fix $\epsilon>0$ arbitrarily. Consider the set $P_\epsilon$ consisting of the union of the periodic orbits $\gamma \subset M\setminus \dot\partial\Sigma$ satisfying  $\rho(\gamma,\Sigma)\geq1-\epsilon$. We need to show that $\overline{P_\epsilon}$ has positive Liouville measure.

Let $\{\gamma^n_j\},\{p^n_j\}$ be a sequence of finite collections of periodic orbits and of weights satisfying
\begin{equation}\label{approximating_periodic_orbits}
\sum_j p^n_j = 1, \ \forall n \qquad \qquad \sum_j p^n_j \frac{(\gamma^n_j)_*{\rm Leb}}{T(\gamma^n_j)} \to \mu_\lambda \ \text{as measures.}
\end{equation}

First we prove the following general \\

\noindent \textbf{Claim.} If $U$ is an open invariant set of full Liouville measure and if we set
\begin{equation*}
J_n = \{ j \mid \gamma^n_j \subset U \}, \qquad L_n =  \sum_{j\in J_n} p^n_j
\end{equation*}
then 
\begin{equation}\label{approx_per_orbits_relative}
\lim_{n\to\infty} L_n = 1 \qquad \lim_{n\to\infty} \sum_{j\in J_n} \hat p^n_j \frac{(\gamma^n_j)_*{\rm Leb}}{T(\gamma^n_j)} = \mu_\lambda \qquad \text{with} \qquad \hat p^n_j = \frac{p^n_j}{L_n}
\end{equation}

\begin{proof}[Proof of Claim.]
To prove~\eqref{approx_per_orbits_relative} first note that by~\eqref{approximating_periodic_orbits} $$ 1 = \mu_\lambda(U) = \lim_{n\to\infty} \sum_{j\in J_n} p^n_j = \lim_{n\to\infty} L_n $$ and then compute for an arbitrary open set $V\subset M$
\begin{equation}
\begin{aligned}
\mu_\lambda(V) &= \mu_\lambda(U\cap V) \\
&= \lim_{n\to\infty} \frac{1}{L_n} \mu_\lambda(U\cap V) \\
&= \lim_{n\to\infty} \frac{1}{L_n} \sum_j \frac{p^n_j}{T(\gamma^n_j)} \ {\rm Leb}(\{t\in\R/T(\gamma^n_j)\Z \mid \gamma^n_j(t) \in U\cap V \}) \\
&= \lim_{n\to\infty} \frac{1}{L_n} \sum_{j\in J_n} \frac{p^n_j}{T(\gamma^n_j)} \ {\rm Leb}(\{t\in\R/T(\gamma^n_j)\Z \mid \gamma^n_j(t) \in U\cap V \}) \\
&= \lim_{n\to\infty} \frac{1}{L_n} \sum_{j\in J_n} \frac{L_n \hat p^n_j}{T(\gamma^n_j)} \ {\rm Leb}(\{t\in\R/T(\gamma^n_j)\Z \mid \gamma^n_j(t) \in U\cap V \}) \\
&= \lim_{n\to\infty} \sum_{j\in J_n} \frac{\hat p^n_j}{T(\gamma^n_j)} \ {\rm Leb}(\{t\in\R/T(\gamma^n_j)\Z \mid \gamma^n_j(t) \in V \})
\end{aligned}
\end{equation}
as desired, proving~\eqref{approx_per_orbits_relative}.
\end{proof}

To prove the theorem we argue by contradiction and assume that $\mu_\lambda(\overline{P_\epsilon})=0$. Hence $U_\epsilon = M \setminus (\overline{P_\epsilon}\cup\dot\partial\Sigma)$ is an open invariant set of full Liouville measure. From above \emph{Claim} applied to $U_\epsilon$ we see that there is no loss of generality to assume that $\gamma^n_j \subset U_\epsilon$ for all $n,j$. The function $f_{\mu_\lambda,\Sigma^*}$ in~\eqref{function_int_number} is defined by the limits
\begin{equation}\label{def_f_0}
f_{\mu_\lambda,\Sigma^*}(p) = \lim_{t_n\to+\infty} \frac{{\rm int}(k(t_n,p),\Sigma)}{t_n}
\end{equation}
on an invariant Borel set $E\subset M\setminus\dot\partial \Sigma$ of full Liouville measure that can be chosen to contain all the periodic orbits in $M\setminus\dot\partial\Sigma$, and satisfies
\begin{equation}\label{app_ac_li_lemma_1}
\int_{M\setminus\dot\partial\Sigma} f_{\mu_\lambda,\Sigma^*} \ d\mu_\lambda = \frac{{\rm int}(\lambda\wedge d\lambda,\Sigma)}{\vol(\lambda)} = \frac{T(\Sigma)}{\vol(\lambda)}
\end{equation}
Here we made use of Lemma~\ref{lem_action_linking}. Moreover, in the proof of Lemma~\ref{lem_action_linking} (Section~\ref{sec_proof_action_link}) it is shown that there exists a smooth closed $1$-form $\beta_\delta$ defined on $M\setminus\dot\partial \Sigma$ such that $i_X\beta_\delta$ is bounded, $\beta_\delta$ represents the class $\Sigma^*$, and
\begin{equation}
\int_{M\setminus\dot\partial\Sigma} f_{\mu_\lambda,\Sigma^*} \ d\mu_\lambda = \int_{M\setminus\dot\partial\Sigma} i_X\beta_\delta \ d\mu_\lambda
\end{equation}
Hence
\begin{equation}
\begin{aligned}
\frac{T(\Sigma)}{\vol(\lambda)} &= \int_{M\setminus\dot\partial\Sigma} f_{\mu_\lambda,\Sigma^*} \ d\mu_\lambda \\
&= \int_{M\setminus\dot\partial\Sigma} i_X\beta_\delta \ d\mu_\lambda \\
&= \int_{U_\epsilon} i_X\beta_\delta \ d\mu_\lambda \\
&= \lim_{n\to\infty} \int_{U_\epsilon} i_X\beta_\delta \ d \left( \sum_j p^n_j \frac{(\gamma^n_j)_*{\rm Leb}}{T(\gamma^n_j)} \right) \\
&= \lim_{n\to\infty} \sum_j p^n_j \ \frac{{\rm int}(\gamma^n_j,\Sigma)}{T(\gamma^n_j)}  \\
&\leq \lim_{n\to\infty} \sum_j p^n_j (1-\epsilon) \frac{T(\Sigma)}{\vol(\lambda)} = (1-\epsilon) \frac{T(\Sigma)}{\vol(\lambda)}
\end{aligned}
\end{equation}
In the third equality we used our contradiction assumption that $U_\epsilon$ has full Liouville measure. In the fifth line we used that $X$ is tangent to $\gamma^n_j$ and $\beta_\delta$ represents $\Sigma^*$. In the sixth line we used the assumption that all $\gamma^n_j$ lie in $U_\epsilon$ and satisfy $\rho(\gamma^n_j,\Sigma)<1-\epsilon$ (by the definition of $U_\epsilon$). This contradiction concludes the proof in this case.

The closure of the set of periodic orbits satisfying $\rho(\cdot,\Sigma)\leq 1+\epsilon$ can be handled analogously.

\section{Reinhardt domains}\label{sec_toric}

We prove Proposition~\ref{prop_toric}. The Reeb vector field $X$ on $(\partial W,\lambda_0)$ coincides with the Hamiltonian vector field $X_H$ defined by $-dH=i_{X_H}\omega_0$, hence 
\[
X(x_1,y_1,x_2,y_2) = 2D_1F(r_1^2,r_2^2)\partial_{\theta_1} + 2D_2F(r_1^2,r_2^2)\partial_{\theta_2}
\]
where $D_j$ denotes a partial derivative in the $j$-th variable. It follows that each torus given by a point in the curve $F=1$ in the interior of the first quadrant of the $(r_1^2,r_2^2)$-plane is invariant by the flow. Such a torus is either foliated by periodic orbits (rational torus) or does not contain periodic orbits (irrational torus). The rational tori are characterised by $\nabla F$ having commensurable coordinates, and there is a unique $(p,q) \in \Z^2$ in the complement of the (closed) third quadrant such that $\gcd(p,q)=1$ and $\nabla F = t (p,q)$ for some $t>0$. We call this torus a $(p,q)$-torus. The primitive period $T$ of the orbits in such a torus is determined by
\[
T = \frac{\pi p}{D_1F} = \frac{\pi q}{D_2F} 
\]
where the partial derivatives of $F$ are evaluated at the corresponding point $(r_1^2,r_2^2)$. In addition to the rational tori, there are two special orbits given by the points of $F=1$ in the $r_1^2$-axis and in the $r_2^2$-axis. We denote the latter by $\gamma_1=\{r_1=0\}$ and the former by $\gamma_2=\{r_2=0\}$. If $\gamma$ is a periodic orbit in a $(p,q)$-torus then 
\begin{equation}
\link(\gamma,\gamma_1)=p \qquad \link(\gamma,\gamma_2)=q
\end{equation}
In which case
\begin{equation}\label{rho_with_gamma_1}
\rho(\gamma,\gamma_1) = \frac{p2\pi^2A}{\frac{\pi p}{D_1F|_{\gamma}}\frac{\pi}{D_2F|_{\gamma_1}}} = 2A \ D_1F|_\gamma D_2F|_{\gamma_1}
\end{equation}
where $A$ is the area of the region in $\{F\leq1\}$ that lies in the first quadrant. Analogously
\begin{equation}\label{rho_with_gamma_2}
\rho(\gamma,\gamma_2) = \frac{q2\pi^2A}{\frac{\pi q}{D_2F|_{\gamma}}\frac{\pi}{D_1F|_{\gamma_2}}} = 2A \ D_2F|_\gamma D_1F|_{\gamma_2}
\end{equation}
We can also compute
\begin{equation}\label{rho_end_points}
\rho(\gamma_1,\gamma_2) = \frac{2\pi^2A}{\frac{\pi}{D_1F|_{\gamma_2}}\frac{\pi}{D_2F|_{\gamma_1}}} = 2A \ D_1F|_{\gamma_2} D_2F|_{\gamma_1}
\end{equation}
Now let $\hat\gamma$ be a periodic orbit on a $(\hat p,\hat q)$-torus. Up to relabeling, we can assume that the $(p,q)$-torus is closer to $\gamma_1$ than the $(\hat p,\hat q)$-torus, in the sense that the $(p,q)$-torus divides $\partial W$ into two solid tori, one of them contains $\gamma_1$ while the other contains $\gamma_2$ and the $(\hat p,\hat q)$-torus. Hence $\gamma$ can be homotoped to $\gamma_1^q$ and $\hat\gamma$ can be homotoped to $\gamma_2^{\hat p}$ by homotopies with disjoint images. Hence $\link(\gamma,\hat\gamma)=\hat pq$. We get
\begin{equation}\label{rho_in_the_complement}
\rho(\gamma,\hat\gamma) = \frac{\hat pq 2\pi^2A}{\frac{\pi q}{D_2F|_\gamma} \frac{\pi\hat p}{D_1F|_{\hat\gamma}}} = 2A \ D_2F|_\gamma D_1F|_{\hat\gamma}
\end{equation}

The first conclusion from the above formulas is that for ellipsoids the number $1$ belongs to the systolic interval, and the systolic norm vanishes.

Secondly, the above calculations show that the systolic norm can only vanish in the ellipsoid case. In fact,~\eqref{rho_with_gamma_1} and the $1$-homogeneity of $F$ together imply that $D_1F$ is constant in the first quadrant if the systolic norm is zero. Analogously,~\eqref{rho_with_gamma_2} and the $1$-homogeneity of $F$ together imply that $D_2F$ is constant in the first quadrant if the systolic norm is zero. Hence $F$ coincides with a linear function in the first quadrant and $W$ is an ellipsoid if the systolic norm vanishes, as desired.

Finally, let us prove for toric domains that $1$ always belongs to the systolic interval. If $\rho(\gamma_1,\gamma_2)=1$ then there is nothing to prove. We proceed assuming $\rho(\gamma_1,\gamma_2)\neq1$. Below we only argue for the case $\rho(\gamma_1,\gamma_2)>1$, since the case $\rho(\gamma_1,\gamma_2)<1$ is analogous.

Let $(x,y)$ denote coordinates in $\R^2$. By $1$-homogeneity of $F$ the domain $\{F\leq1\}$ is star-shaped in $\R^2$ and we can parametrize the part of its boundary $\{F=1\}$ in the first quadrant as a curve $$ c:\theta \in [0,\pi/2] \mapsto c(\theta) \in \{F=1\}\cap\{x\geq0,y\geq0\} $$ where $\theta$ is the polar angle. Let $S(\theta)$  be the area of the intersection of the domain $\{F\leq1\}$ with the sector $\{0\leq\arctan (y/x)\leq\theta\}$, and set $t(\theta):=2S(\theta)$. The star-shapedness of the domain implies that $t'(\theta)>0$. Hence we can invert the function $t(\theta)$ on $\{\theta\in[0,\pi/2]\}$ to obtain a function $\theta:[0,2A]\to[0,\pi/2]$. Denote $C(t) = c(\theta(t))=(x(t),y(t))$, $t\in[0,2A]$. Let $a,b>0$ be determined by $C(0)=(a,0)$, $C(2A)=(0,b)$.  Note that 
\begin{equation*}
S(\theta(t)) = \frac{1}{2} \int_0^t x(\tau)y'(\tau)-y(\tau)x'(\tau) \ d\tau.
\end{equation*}
Hence
\begin{equation}\label{infinitesimal_area}
1 = 2 \frac{d}{dt} S(\theta(t)) = 
x(t)y'(t)-y(t)x'(t).
\end{equation}
By $1$-homogeneity of $F$ we have at the point $(x(t),y(t))$
\begin{equation}\label{1-hom_differential_id}
1 = F = xD_1F + yD_2F.
\end{equation}
Note that $(y'(t),-x'(t))$ is a positive multiple of $(D_1F,D_2F)$ evaluated at $(x(t),y(t))$. By~\eqref{infinitesimal_area} and~\eqref{1-hom_differential_id} the vector $(y'(t),-x'(t))$ must coincide with $(D_1F,D_2F)$ evaluated at $(x(t),y(t))$ since both these vectors are not parallel to $(x(t),y(t))$. Hence $C(t)$ solves Hamilton's equation
\begin{equation}\label{Hamilton_toric}
x'(t) = -D_2F(x(t),y(t)), \qquad y'(t) = D_1F(x(t),y(t)).
\end{equation}

We make another use of~\eqref{1-hom_differential_id} to compute
\begin{equation}
\begin{aligned}
1 = F(C(0)) = aD_1F(a,0) &\Rightarrow D_1F(a,0) = \frac{1}{a} \\
1 = F(C(2A)) = bD_2F(0,b) &\Rightarrow D_2F(0,b) = \frac{1}{b}
\end{aligned}
\end{equation}
From~\eqref{rho_end_points} and our standing assumption $\rho(\gamma_1,\gamma_2)>1$ we get
\begin{equation}\label{standing_toric}
1 < \rho(\gamma_1,\gamma_2) = \frac{2A}{ab}
\end{equation}
Denoting by $(t_1,t_2)$ the coordinates in $[0,2A]\times[0,2A]$, we compute the average:
\begin{equation}
\begin{aligned}
& \frac{1}{4A^2} \int_{[0,2A]^2} D_1F(x(t_1),y(t_1)) D_2F(x(t_2),y(t_2)) \ dt_1dt_2 \\
&= \frac{1}{4A^2} \int_{[0,2A]^2} -y'(t_1)x'(t_2) \ dt_1dt_2 \\
&= \frac{1}{4A^2} \left( \int_0^{2A} -x'(t_2)dt_2 \right) \left( \int_0^{2A} y'(t_1)dt_1 \right) \\
&= \frac{1}{4A^2} \left( x(0)-x(2A) \right) \left( y(2A)-y(0) \right) = \frac{1}{2A} \frac{ab}{2A} < \frac{1}{2A}
\end{aligned}
\end{equation}
where~\eqref{standing_toric} was used for the last inequality. Hence we find $(t_1^*,t_2^*)\in[0,2A]^2$ such that $$ 2A \ D_1F(x(t_1^*),y(t_1^*)) \ D_2F(x(t_2^*),y(t_2^*)) < 1 $$ For each $j=1,2$ the value $t_j^*$ corresponds to an invariant torus or to one of the end orbits $\gamma_1,\gamma_2$. By continuity of the partial derivatives of $F$ in $\R^2\setminus\{(0,0)\}$ and by formulas~\eqref{rho_with_gamma_1}-\eqref{rho_in_the_complement} we find two rational tori with orbits $\gamma,\hat\gamma$ satisfying $\rho(\gamma,\hat\gamma)<1$. Hence $\rho(\gamma,\hat\gamma)<1<\rho(\gamma_1,\gamma_2)$, from where we conclude that $1$ belongs to the systolic interval.

\begin{remark}\label{rmk_rot_numbers}
The Seifert rotation number $\mathrm{rot}_0(\gamma)$ is defined to be the rotation number of the transverse linearised dynamics in a Seifert framing. For details see~\cite[subsection~2.2]{SFS}. 
In the notation above, if $\hat\gamma \to \gamma$ in the sense that the rational torus of $\hat\gamma$ approaches that of $\gamma$, then $\frac{\link(\hat\gamma,\gamma)}{\hat T/T} \to \mathrm{rot}_0(\gamma)$, and hence $$ \rho(\gamma,\hat\gamma) = \frac{\link(\hat\gamma,\gamma)\vol}{T\hat T} = \frac{\link(\hat\gamma,\gamma)}{\hat T/T} \frac{\vol}{T^2} \to \mathrm{rot}_0(\gamma) \frac{\vol}{T^2}. $$ Then, using the above formulas for Reinhardt domains, we see that the limit of $\rho(\gamma,\hat\gamma) = 2A \ D_1F|_{\hat\gamma} D_2F|_\gamma$ as $\hat\gamma \to \gamma$ is $\mathrm{rot}_0(\gamma) \vol \ T^{-2} = 2A \ D_1F|_\gamma D_2F|_\gamma$. It follows that the systolic interval and its enlarged version coincide with the interval made of values $2A \ D_1F(x_1(t),x_2(t)) D_2F(x_1(\hat t),x_2(\hat t))$ where $t,\hat t$ vary on $[0,2A]$ independently, as claimed in Remark~\ref{rmk_H}.
\end{remark}

\end{document}